\documentclass[a4paper,11pt]{article}
\usepackage[T1]{fontenc}
\usepackage{lmodern}
\usepackage[english]{babel}
\usepackage{amsmath,amssymb,amsthm,mathtools}
\usepackage{microtype}
\usepackage{needspace}
\usepackage[a4paper,textwidth=160mm,textheight=235mm,top=27mm,headheight=14pt]{geometry}
\usepackage{fancyhdr}
\usepackage[hidelinks]{hyperref}
\allowdisplaybreaks[3]
\newtheorem{theorem}{Theorem}[section]
\newtheorem{proposition}[theorem]{Proposition}
\newtheorem{lemma}[theorem]{Lemma}
\newtheorem{corollary}[theorem]{Corollary}
\theoremstyle{remark}
\newtheorem{remark}[theorem]{Remark}
\newtheorem{example}[theorem]{Example}
\numberwithin{equation}{section}
\newcommand{\R}{\mathbb R}
\newcommand{\C}{\mathbb C}
\newcommand{\D}{\mathcal D}
\newcommand{\Q}{\mathcal Q}
\newcommand{\Row}{\mathcal R}
\newcommand{\rank}{\operatorname{rank}}
\newcommand{\diag}{\operatorname{diag}}
\newcommand{\spanof}{\operatorname{span}}

\newcommand{\controlcost}{\mathfrak N_\infty}
\newcommand{\ind}{\mathbf 1}
\newcommand{\norm}[1]{\left\lVert#1\right\rVert}
\newcommand{\step}[1]{\par\medskip\Needspace{4\baselineskip}\noindent\emph{#1}\par\smallskip\noindent}
\title{Sharp Observability Costs for Strongly Coupled Parabolic Systems
with Vanishing Principal Coupling\thanks{This work was partially supported by the National Natural Science Foundation of China under Grant No.~12601883.}}
\author{Xiaomin Zhu\thanks{Civil Aviation Flight University of China, Deyang 618307, China (xiaominzhu123@yeah.net).}}
\date{}

\begin{document}
\maketitle
\begin{abstract}
We study the loss of observability for a class of strongly coupled parabolic
systems as the coupling parameter tends to zero. The system is observed
through a fixed matrix $B$ on a measurable subset of space--time with
positive measure. Under the Kalman rank condition, we determine the exact blow-up rates of the optimal observability cost. Different components of the state, as well as different combinations of them, may have different blow-up rates, and these rates are determined by the algebraic structure of the pair $(B,K)$. The full-state observability
cost is determined by the largest of these rates. We also show that the
worst-case minimum $L^\infty$-norm of null controls has the same
asymptotic behavior.
\end{abstract}



\section{Introduction}\label{sec:introduction}

Let $\Omega\subset\R^d$ be a bounded, connected, Lipschitz and locally
star-shaped domain. We use local star-shapedness in the sense of
\cite[Definition~4]{ApraizEscauriazaWangZhang2014}.
Let $T>0$ and $a>0$. Fix integers $n\ge2$ and $1\le m\le n$, and matrices
\[
K\in\R^{n\times n},\qquad K^\top=-K,\qquad
B\in\R^{m\times n},\qquad \rank B=m.
\]
For $\varepsilon\in\R$, consider
\begin{equation}\label{eq:original-system}
\begin{cases}
\partial_t y_\varepsilon=(aI_n+\varepsilon K)\Delta y_\varepsilon
   &\text{in }\Omega\times(0,T),\\
y_\varepsilon=0&\text{on }\partial\Omega\times(0,T),\\
y_\varepsilon(0)=y_0&\text{in }\Omega.
\end{cases}
\end{equation}
The symmetric part of the diffusion matrix is $aI_n$. Thus the system
is uniformly parabolic. For each $y_0\in L^2(\Omega;\R^n)$, it has a
unique mild solution in $C([0,T];L^2(\Omega;\R^n))$, denoted by
$y_\varepsilon(\cdot;y_0)$.
Throughout the paper, $|\cdot|$ is the Euclidean norm and
$\norm{\cdot}_2$ is the $L^2(\Omega)$ norm, with the component dimension
specified by the argument. We use $\norm{\cdot}$ for the induced norm of a matrix or a bounded operator.

Let $\D\subset\Omega\times(0,T)$ be Lebesgue measurable, with $|\D|>0$.
The measured quantity is $By_\varepsilon$ on $\D$. For $\varepsilon\ne0$, let $C_{\rm full}(\varepsilon;T,\D)$ be the
least constant in
\begin{equation}\label{eq:full-cost}
\norm{y_\varepsilon(T;y_0)}_2
\le C_{\rm full}(\varepsilon;T,\D)\int_\D |By_\varepsilon(x,t;y_0)|\,dx\,dt
\qquad\text{for all }y_0\in L^2(\Omega;\R^n).
\end{equation}
For a nonzero row $\ell\in\R^{1\times n}$, define
$C_\ell(\varepsilon;T,\D)$ in the same way by
\begin{equation}\label{eq:functional-cost}
\norm{\ell y_\varepsilon(T;y_0)}_2
\le C_\ell(\varepsilon;T,\D)\int_\D |By_\varepsilon(x,t;y_0)|\,dx\,dt
\qquad\text{for all }y_0\in L^2(\Omega;\R^n).
\end{equation}
Here $\ell y_\varepsilon(T)$ is a scalar-valued function on $\Omega$.
We set the cost equal to $+\infty$ when there is no finite admissible
constant. Our aim is to determine the orders of these costs as
$\varepsilon\to0$. When $T$ and $\D$ are fixed, we omit them from the
cost notation.

At $\varepsilon=0$, every component of \eqref{eq:original-system}
satisfies the scalar heat equation independently. When $m<n$, choose
$v\in\ker B\setminus\{0\}$ and a nonzero $\phi\in L^2(\Omega)$.
The initial state $y_0(x)=\phi(x)v$ gives
\[
y(x,t)=\bigl(e^{at\Delta}\phi\bigr)(x)v,
\qquad By(x,t)=0.
\]
The state remains nonzero at every finite positive time. Thus full-state observability fails at zero coupling when $m<n$. For every $\varepsilon$, the symmetric part of the diffusion matrix is $aI_n$. As $\varepsilon\to0$, the coupling term $\varepsilon K\Delta y_\varepsilon$, which transfers information between different component directions, vanishes. This leads to a quantitative question: at what rate does the full-state observability cost blow up as the coupling tends to zero?

The complex heat equation observed through its real part gives a simple illustration. Write
$z_\varepsilon=y_1+iy_2$ in
\[
\partial_t z_\varepsilon=(a+i\varepsilon)\Delta z_\varepsilon,
\qquad z_\varepsilon|_{\partial\Omega}=0.
\]
Then its real and imaginary parts satisfy \eqref{eq:original-system} with
$K=\begin{pmatrix}0&-1\\1&0\end{pmatrix}$ and $B=(1,0)$.
Let $-\Delta\phi_1=\lambda_1\phi_1$ be the first Dirichlet eigenpair,
with $\norm{\phi_1}_2=1$. For $z_0=i\phi_1$, the solution is
\begin{equation}\label{eq:intro-complex-mode}
z_\varepsilon(x,t)
=e^{-a\lambda_1t}
\bigl(\sin(\varepsilon\lambda_1t)+i\cos(\varepsilon\lambda_1t)\bigr)
\phi_1(x).
\end{equation}
Its terminal norm is
$
\|z_\varepsilon(T)\|_2=e^{-a\lambda_1T}.
$
On the other hand, since only the real part is observed,
\[
\operatorname{Re}z_\varepsilon(x,t)
=e^{-a\lambda_1t}\sin(\varepsilon\lambda_1t)\phi_1(x).
\]
Using $|\sin s|\le |s|$, we obtain
\[
\int_{\D}|\operatorname{Re}z_\varepsilon(x,t)|\,dx\,dt
\le
|\varepsilon|\lambda_1
\int_{\D}t e^{-a\lambda_1t}|\phi_1(x)|\,dx\,dt
\le C_{\D}|\varepsilon|,
\]
where $C_{\D}>0$ is independent of $\varepsilon$. Therefore, if
\[
\|z_\varepsilon(T)\|_2
\le
C_{\mathrm{Re}}(\varepsilon;T,\D)
\int_{\D}|\operatorname{Re}z_\varepsilon|\,dx\,dt
\]
holds for all initial data, then the choice $z_0=i\phi_1$ gives
\[
e^{-a\lambda_1T}
\le
C_{\D}C_{\mathrm{Re}}(\varepsilon;T,\D)|\varepsilon|.
\]
Thus the full-state observability cost is at least of order
$|\varepsilon|^{-1}$ as $\varepsilon\to0$.
This simple example suggests a more general question. For systems with
several components, information may be transferred to the observed
components through the coupling structure. It is therefore natural to ask how the observability cost is determined by the algebraic structure of $(B,K)$ and how it behaves as the coupling
parameter tends to zero.

This question is closely related to algebraic conditions for the control
and observation of coupled parabolic systems.
Ammar-Khodja et al. \cite{AmmarKhodja2005} proved sufficient conditions
for null controllability of parabolic systems by one control force.
A Kalman rank condition for localized distributed control with a
diagonalizable diffusion matrix was proved in
\cite{AmmarKhodja2009}. Lissy and Zuazua \cite{LissyZuazua2019}
characterized internal observability for systems with constant strongly
elliptic diffusion matrices through a rank condition for each spatial
mode. For $aI_n+\varepsilon K$ and fixed $\varepsilon\ne0$, this
condition reduces to the Kalman condition for $(B,K)$.

The algebraic structure also appears in quantitative estimates of control
and observability costs. For finite-dimensional linear systems,
Seidman \cite{Seidman1988} obtained sharp small-time estimates for the
worst-case minimum $L^2$ norm of controls, with the exponent determined
by the Kalman structure. In a PDE setting, Lasiecka and Seidman
\cite{LasieckaSeidman2006} studied observability costs for a thermoelastic
system when one component is observed. Their estimates include the
regime in which the coupling between the thermal and elastic variables
tends to zero. These results indicate that quantitative cost estimates may depend on both the algebraic structure of the system and the size of the coupling parameter.

We also consider observation on measurable subsets of space--time, which has been studied in several settings. Phung and Wang \cite{PhungWang2013} considered parabolic equations with observation on a fixed spatial subset and a measurable time set. Apraiz et al. \cite{ApraizEscauriazaWangZhang2014} proved observability for the heat equation from arbitrary measurable subsets of space--time with positive measure, while Wang and Zhang \cite{WangZhang2017} obtained related results for abstract evolution equations. For the two-component system, Fu et al. \cite{FuWangYuZhu2026} proved single-component $L^1$ observability on measurable subsets of space--time for fixed nonzero coupling.

The main contribution of this paper is a quantitative description of the observability cost as the coupling parameter $\varepsilon$ tends to zero. We show that different components of the state, or different linear combinations of them, may have different blow-up rates. These rates are determined by the algebraic structure of $B$ and $K$. In particular, the full-state observability cost is determined by the largest blow-up rate. All these estimates are sharp and remain valid for observation on any measurable subset of space--time with positive measure.

Moreover, we study the associated null control problem. The worst-case minimum $L^\infty$-norm of null controls has the same blow-up rate as the full-state observability cost. In particular, for the complex heat equation observed through its real part, the optimal observability cost is of order $|\varepsilon|^{-1}$ as $\varepsilon\to0$.

The rest of the paper is organized as follows.
Section~\ref{sec:modal} introduces the algebraic structure associated with
$B$ and $K$ and states the main observability estimates.
Section~\ref{sec:measurable} establishes observability on measurable
subsets of space--time.
Section~\ref{sec:sharpness} proves the sharp bounds for the observability
costs and gives several examples.
Section~\ref{sec:control} studies the corresponding minimum-norm null
control problem.

\section{Kalman rank condition and a uniform modal estimate}\label{sec:modal}

In this section, we describe the algebraic structure that determines the
blow-up rate of the observability cost. The main point is that different
state directions may reach the observed components after different numbers
of couplings through $K$. These different levels lead to different powers
of $\varepsilon$ as $\varepsilon\to0$. The main result of this section is a uniform estimate for each spatial eigenmode. It will be used in Section~\ref{sec:measurable} to obtain observability from the measurable set $\D$.

We first explain how the coupling matrix $K$ enters the observation and why its algebraic structure affects the observability cost as $\varepsilon\to0$. For this purpose, consider a single Dirichlet eigenfunction. Let
$
-\Delta\phi=\lambda\phi
$
and take initial data of the form $y_0(x)=\phi(x)v$, where
$v\in\R^n$. The corresponding solution has the form
$$
y_\varepsilon(x,t)=\phi(x)z_\varepsilon(t),
$$
where
$$
z_\varepsilon'(t)
=-\lambda(aI_n+\varepsilon K)z_\varepsilon(t)\quad \text{and} \quad z_\varepsilon(0)=v.
$$
Hence
$$
Bz_\varepsilon(t)
=
e^{-a\lambda t}Be^{-\varepsilon\lambda tK}v.
$$

Expanding the matrix exponential gives
$$
Bz_\varepsilon(t)
=
e^{-a\lambda t}
\sum_{r=0}^{\infty}
\frac{(-\varepsilon\lambda t)^r}{r!}BK^rv.
$$
Therefore, the observation involves not only $B$, but also the successive
matrices
$
BK,\ BK^2,\ \ldots,
$
and the contribution of $BK^r$ is accompanied by the factor
$\varepsilon^r$.

The expansion above suggests grouping the state directions according to the number of couplings needed to reach the observation. For $r\ge0$, set
\begin{equation}\label{eq:row-filtration}
\Row_{-1}=\{0\},\qquad
\Row_r=\spanof\{\text{rows of }B,BK,\ldots,BK^r\}.
\end{equation}
These spaces form an increasing sequence
$
\Row_{-1}\subset\Row_0\subset\Row_1\subset\cdots.
$
The Kalman rank condition is
\begin{equation}\label{eq:kalman-rank}
\rank\begin{pmatrix}B\\BK\\\vdots\\BK^{n-1}\end{pmatrix}=n.
\end{equation}
The condition \eqref{eq:kalman-rank} describes the algebraic accessibility of the state directions through $B$ and $K$. Its role in observability from the measurable set $\D$ will be established in Section~\ref{sec:measurable}. When \eqref{eq:kalman-rank} holds, define
\begin{equation}\label{eq:depths}
q=\min\{r:\Row_r=\R^{1\times n}\},\qquad
\rho(\ell)=\min\{r:\ell\in\Row_r\}\quad(\ell\ne0).
\end{equation}
Here, \(\rho(\ell)\) is the first Kalman level containing the direction \(\ell\), while \(q\) is the largest depth needed to recover the whole state space. Under the Kalman rank condition, the inclusions are strict until the whole space is reached; hence $q\le n-m$. In particular,
$
0\le\rho(\ell)\le q\le n-m\le n-1,
$
and $\rho(\ell)=0$ exactly when $\ell$ is a linear combination of the
rows of $B$.

The next theorem shows that these depths give the exact powers of the
observability costs defined in \eqref{eq:full-cost} and
\eqref{eq:functional-cost}.
\begin{theorem}\label{thm:sharp}
If \eqref{eq:kalman-rank} holds, then for every nonzero row $\ell$, there are
constants $c_\ell,M_\ell,\varepsilon_\ell>0$ such that
\begin{equation}\label{eq:functional-sharp}
c_\ell|\varepsilon|^{-\rho(\ell)}
\le C_\ell(\varepsilon;T,\D)
\le M_\ell|\varepsilon|^{-\rho(\ell)},
\qquad 0<|\varepsilon|\le\varepsilon_\ell.
\end{equation}
There are also constants $c,C,\varepsilon_*>0$ such that
\begin{equation}\label{eq:full-sharp}
c|\varepsilon|^{-q}
\le C_{\rm full}(\varepsilon;T,\D)
\le C|\varepsilon|^{-q},
\qquad 0<|\varepsilon|\le\varepsilon_*.
\end{equation}
Moreover, for every fixed $\varepsilon\ne0$, full-state observability holds if and only if \eqref{eq:kalman-rank} is satisfied. If the condition fails then $C_{\rm full}(\varepsilon;T,\D)=+\infty$ for every $\varepsilon\ne0$.
\end{theorem}

\begin{remark}
Theorem~\ref{thm:sharp} shows that the Kalman depth determines the exact
blow-up order of the observability cost as $\varepsilon\to0$. For a fixed
row $\ell$, the cost of recovering $\ell y_\varepsilon(T)$ is of order
$|\varepsilon|^{-\rho(\ell)}$, while the full-state cost is of order
$|\varepsilon|^{-q}$.

Thus different state directions may have different observability costs,
according to their Kalman depths. Moreover, for every fixed
$\varepsilon\ne0$, full-state observability holds if and only if
\eqref{eq:kalman-rank} is satisfied. Here and below,
$f(\varepsilon)\asymp g(\varepsilon)$ means two-sided bounds by positive
constants for all sufficiently small nonzero $\varepsilon$.
\end{remark}

The proof of Theorem~\ref{thm:sharp} is completed in Section~\ref{sec:sharpness}.
The upper bounds in Theorem~\ref{thm:sharp} follow from the following weighted observability estimate.
\begin{theorem}\label{thm:weighted}
Assume \eqref{eq:kalman-rank}. For every $\varepsilon_0>0$, there is
$C=C(\Omega,a,T,\D,K,B,\varepsilon_0)>0$ such that
\begin{equation}\label{eq:weighted-observability}
\left(\sum_{r=0}^q |\varepsilon|^{2r}
\norm{BK^r y_\varepsilon(T;y_0)}_2^2\right)^{1/2}
\le C\int_\D |By_\varepsilon(x,t;y_0)|\,dx\,dt
\end{equation}
for every $y_0\in L^2(\Omega;\R^n)$ and
$|\varepsilon|\le\varepsilon_0$. At $\varepsilon=0$, the term with
$r=0$ has weight one and all other terms have weight zero.
\end{theorem}

Section~\ref{sec:measurable} proves Theorem~\ref{thm:weighted}. We first
introduce coordinates adapted to the Kalman depths and establish the
uniform modal estimate needed there.

Write $b_1,\ldots,b_m$ for the rows of $B$. Starting from these rows,
extend them successively to bases of
$$
\Row_0,\Row_1,\ldots,\Row_q=\R^{1\times n}
$$
by selecting rows from $BK^r$ at level $r$. Denote the selected rows by
\begin{equation}\label{eq:adapted-basis}
p_i=b_{\alpha_i}K^{d_i}, \quad \alpha_i\in \{1,\ldots,m\}
\end{equation}
where $(\alpha_i,d_i)=(i,0) \ (1\le i\le m)$ and
$0=d_1=\cdots=d_m\le\cdots\le d_n=q$.
For every $r$, the rows with $d_i\le r$ form a basis of $\Row_r$.
Hence
$
\rho(p_i)=d_i.
$
Let $P$ be the matrix with rows $p_1,\ldots,p_n$. Thus the coordinates
defined by $P$ are arranged according to their Kalman depths.

We next attach to each coordinate the power of $\varepsilon$
corresponding to its depth. Set
\begin{equation}\label{eq:weighted-variables}
D_\varepsilon
=\diag(\varepsilon^{d_1},\ldots,\varepsilon^{d_n}),
\qquad
w_\varepsilon=D_\varepsilon Py_\varepsilon,
\qquad
E=\begin{pmatrix}I_m&0\end{pmatrix},
\end{equation}
where $\varepsilon^0=1$ also at $\varepsilon=0$. Since the first $m$
rows of $P$ are the rows of $B$ and have depth zero,
\begin{equation}\label{eq:weighted-output}
Ew_\varepsilon=By_\varepsilon,
\qquad
\norm{w_\varepsilon(T)}_2^2
=
\sum_{i=1}^n
|\varepsilon|^{2d_i}
\norm{p_i y_\varepsilon(T)}_2^2.
\end{equation}

The norm of $w_\varepsilon$ is equivalent to the weighted quantity in
Theorem~\ref{thm:weighted}. More precisely,
\begin{equation}\label{eq:weight-comparison}
|D_\varepsilon Pv|
\le
\left(
\sum_{r=0}^q
|\varepsilon|^{2r}|BK^rv|^2
\right)^{1/2}
\le
C|D_\varepsilon Pv|,
\qquad
|\varepsilon|\le\varepsilon_0,\quad v\in\R^n.
\end{equation}
Indeed, each $p_i$ is one of the rows appearing in the middle term.
Conversely, every row $b_kK^r$ belongs to $\Row_r$, so
$$
b_kK^r
=
\sum_{d_i\le r}c_{ki}^{(r)}p_i.
$$
Therefore,
$$
\varepsilon^r b_kK^rv
=
\sum_{d_i\le r}
c_{ki}^{(r)}
\varepsilon^{r-d_i}
(D_\varepsilon Pv)_i.
$$
Since $d_i\le r$, the factors $\varepsilon^{r-d_i}$ remain uniformly
bounded for $|\varepsilon|\le\varepsilon_0$. This proves
\eqref{eq:weight-comparison}.

Thus, the weighted quantity in Theorem~\ref{thm:weighted} is equivalent to the standard norm of $w_\varepsilon$, and it remains to prove a uniform observability estimate for $w_\varepsilon$.

 For $\varepsilon\ne0$, set
\begin{equation}\label{eq:normalized-matrix}
A=PKP^{-1},
\qquad
N_\varepsilon
=
D_\varepsilon(\varepsilon A)D_\varepsilon^{-1}.
\end{equation}
Then $w_\varepsilon$ satisfies
\begin{equation}\label{eq:normalized-system}
\partial_t w_\varepsilon
=
(aI_n+N_\varepsilon)\Delta w_\varepsilon,
\qquad
w_\varepsilon|_{\partial\Omega}=0.
\end{equation}
The weighted coordinates $w_\varepsilon$ are chosen so that the transformed matrix $N_\varepsilon$ remains
regular as $\varepsilon\to0$. The following lemma gives the properties
of $N_\varepsilon$ needed in the uniform modal estimate.
\begin{lemma}\label{lem:matrices}
The entries of $N_\varepsilon$ extend polynomially to all real
$\varepsilon$, with
\begin{equation}\label{eq:matrix-entries}
(N_\varepsilon)_{ij}=
\begin{cases}
A_{ij}\varepsilon^{1+d_i-d_j},& d_j\le d_i+1,\\
0, & d_j>d_i+1.
\end{cases}
\end{equation}
For every $v\in\R^n$ and $\varepsilon\in\R$,
\begin{equation}\label{eq:coordinate-recovery}
(EN_\varepsilon^{d_i}v)_{\alpha_i}=v_i,
\qquad i=1,\ldots,n.
\end{equation}
At $\varepsilon=0$,
\begin{equation}\label{eq:nilpotent-limit}
N_0^{q+1}=0,\qquad N_0^q\ne0.
\end{equation}
For every $s>0$,
\begin{equation}\label{eq:time-scaling}
sD_sN_\varepsilon D_s^{-1}=N_{s\varepsilon}.
\end{equation}
Finally, for every $\varepsilon_0>0$, there is $C>0$ such that
\begin{equation}\label{eq:matrix-decay}
\norm{e^{-tN_\varepsilon}}
\le C(1+t^{n-1}),
\qquad
\norm{e^{-t(aI_n+N_\varepsilon)}}
\le Ce^{-at/2}
\end{equation}
for all $t\ge0$ and $|\varepsilon|\le\varepsilon_0$.
\end{lemma}
\begin{remark}
The first relation shows that $N_\varepsilon$ remains well defined as
$\varepsilon\to0$. The second one describes how each coordinate is
recovered through the observation according to its Kalman depth. The
nilpotency of $N_0$ will be used in the zero-coupling limit. The scaling
identity is used for short time intervals, while the last estimate gives
the uniform semigroup bounds needed below.
\end{remark}
\begin{proof}
\step{Step 1. Algebraic properties of $N_\varepsilon$.}
From
$
A=PKP^{-1},
$
we obtain $PK=AP$. Taking the $i$th row gives
$$
p_iK=\sum_{j=1}^n A_{ij}p_j.
$$
By the construction of the adapted basis,
$p_i=b_{\alpha_i}K^{d_i}$, and hence
$$
p_iK=b_{\alpha_i}K^{d_i+1}\in\Row_{d_i+1}.
$$
Since the rows with $d_j\le d_i+1$ form a basis of
$\Row_{d_i+1}$, it follows that
$A_{ij}=0$ whenever $d_j>d_i+1$.
Using
$
N_\varepsilon
=
D_\varepsilon(\varepsilon A)D_\varepsilon^{-1},
$
we obtain
$
(N_\varepsilon)_{ij}
=
A_{ij}\varepsilon^{1+d_i-d_j}
$
when $d_j\le d_i+1$, while the remaining entries vanish. This proves
\eqref{eq:matrix-entries}. In particular, all powers of $\varepsilon$
that occur are nonnegative integers, so $N_\varepsilon$ extends
polynomially to $\varepsilon=0$.

We next prove \eqref{eq:coordinate-recovery}. For $\varepsilon\ne0$,
$
N_\varepsilon^r
=
\varepsilon^rD_\varepsilon A^rD_\varepsilon^{-1}.
$
Using $ED_\varepsilon=E$, $EP=B$, and
$A^r=PK^rP^{-1}$, we obtain
\begin{equation}\label{eq:observed-powers}
EN_\varepsilon^r
=
\varepsilon^rBK^rP^{-1}D_\varepsilon^{-1},
\qquad r\ge0.
\end{equation}
Take $r=d_i$. Since
$p_i=b_{\alpha_i}K^{d_i}$, the $\alpha_i$th row of $BK^{d_i}$ is
$p_i$. Hence the $\alpha_i$th row of the right-hand side of
\eqref{eq:observed-powers} is
$
\varepsilon^{d_i}p_iP^{-1}D_\varepsilon^{-1}.
$
Since $p_i$ is the $i$th row of $P$,
$
p_iP^{-1}=e_i^\top,
$
and therefore
$
\varepsilon^{d_i}e_i^\top D_\varepsilon^{-1}=e_i^\top.
$
Thus
$$
(EN_\varepsilon^{d_i}v)_{\alpha_i}=v_i.
$$

Since both sides depend polynomially on $\varepsilon$, the identity
also holds at $\varepsilon=0$.
At $\varepsilon=0$, \eqref{eq:matrix-entries} shows that
$(N_0)_{ij}$ can be nonzero only when
$
d_j=d_i+1.
$
Consequently, if $(N_0^k)_{ij}\ne0$, then
$
d_j=d_i+k.
$
Since all depths lie between $0$ and $q$, this gives
$
N_0^{q+1}=0.
$
On the other hand, choose $i$ with $d_i=q$. By
\eqref{eq:coordinate-recovery},
$
(EN_0^qv)_{\alpha_i}=v_i,
$
and hence $N_0^q\ne0$. This proves \eqref{eq:nilpotent-limit}.

Finally, for every nonzero entry,
$$
\begin{aligned}
(sD_sN_\varepsilon D_s^{-1})_{ij}
=
A_{ij}s^{1+d_i-d_j}
\varepsilon^{1+d_i-d_j}
=
A_{ij}(s\varepsilon)^{1+d_i-d_j}.
\end{aligned}
$$
By \eqref{eq:matrix-entries}, this is the $(i,j)$ entry of
$N_{s\varepsilon}$. Therefore,
$$
sD_sN_\varepsilon D_s^{-1}=N_{s\varepsilon}.
$$

\step{Step 2. Uniform bounds for the matrix exponential.}
For $\varepsilon\ne0$, $N_\varepsilon$ is similar to $\varepsilon K$.
Since $K$ is skew-symmetric, all eigenvalues of $N_\varepsilon$ are
purely imaginary. At $\varepsilon=0$, \eqref{eq:nilpotent-limit} shows
that the only eigenvalue of $N_0$ is zero.

For each $|\varepsilon|\le\varepsilon_0$, take a unitary Schur
factorization
$$
U^*N_\varepsilon U=\Lambda+R,
$$
where $\Lambda$ is diagonal with purely imaginary entries and $R$ is
strictly upper triangular. Since the family $N_\varepsilon$ is uniformly
bounded on $[-\varepsilon_0,\varepsilon_0]$, there is $M>0$, independent
of $\varepsilon$, such that $\|R\|\le M$.

To estimate $e^{-t(\Lambda+R)}$, set
$
W(t)=e^{t\Lambda}e^{-t(\Lambda+R)}.
$
Then
$$
W'(t)=-R(t)W(t),\qquad W(0)=I,
$$
where
$$
R(t)=e^{t\Lambda}Re^{-t\Lambda}.
$$
Since $e^{t\Lambda}$ is unitary, we have
$
\|R(t)\|=\|R\|\le M.
$
Moreover, every $R(t)$ is strictly upper triangular, so the product of
any $n$ such matrices is zero.
Integrating the equation for $W$ gives
$$
W(t)=I-\int_0^t R(s)W(s)\,ds.
$$
Iterating this identity, the expansion terminates after $n-1$ steps and
gives
$$
W(t)
=
I+
\sum_{k=1}^{n-1}(-1)^k
\int_{0<s_k<\cdots<s_1<t}
R(s_1)\cdots R(s_k)\,ds_k\cdots ds_1.
$$
Using $\|R(s)\|\le M$ and the fact that the simplex
$
0<s_k<\cdots<s_1<t
$
has measure $t^k/k!$, we obtain
$$
\|W(t)\|
\le
\sum_{k=0}^{n-1}\frac{(Mt)^k}{k!}
\le
C(1+t^{n-1}).
$$

Since
$
e^{-tN_\varepsilon}
=
Ue^{-t\Lambda}W(t)U^*
$
and both $U$ and $e^{-t\Lambda}$ are unitary, we have
$$
\|e^{-tN_\varepsilon}\|
\le
C(1+t^{n-1}).
$$
Finally,
$$
\|e^{-t(aI_n+N_\varepsilon)}\|
=
e^{-at}\|e^{-tN_\varepsilon}\|
\le
Ce^{-at}(1+t^{n-1})
\le
Ce^{-at/2}.
$$
This proves \eqref{eq:matrix-decay}.
\end{proof}

We shall also use
\begin{equation}\label{eq:intertwining}
N_\varepsilon D_\varepsilon P
=\varepsilon D_\varepsilon PK,
\qquad
ED_\varepsilon P=B,
\qquad \varepsilon\in\R.
\end{equation}
For $\varepsilon\ne0$, these identities follow directly from
\eqref{eq:normalized-matrix}; the case $\varepsilon=0$ follows by
polynomial continuity. They show that the transformation
$w_\varepsilon=D_\varepsilon Py_\varepsilon$ is compatible with both
the equation and the observation.

For a Dirichlet eigenfunction with eigenvalue $\lambda$,
\eqref{eq:normalized-system} reduces to
$$
v'(t)=-\lambda(aI_n+N_\varepsilon)v(t).
$$
We next establish a uniform observability estimate for the single spatial mode from the observed components on a measurable time set.
\begin{proposition}\label{prop:modal}
Fix $\varepsilon_0>0$. There is $C>0$, depending only on
$a,B,K,\varepsilon_0$, such that
\begin{equation}\label{eq:modal-estimate}
\left|e^{-\lambda(aI_n+N_\varepsilon)\delta}v\right|
\le C\frac{\delta^{n-1}\bigl(1+(\lambda\delta)^{-q}\bigr)}{|F|^n}
\int_F\left|Ee^{-\lambda(aI_n+N_\varepsilon)t}v\right|\,dt
\end{equation}
for every $|\varepsilon|\le\varepsilon_0$, $\lambda,\delta>0$,
$v\in\R^n$, and measurable $F\subset(0,\delta)$ with $|F|>0$.
\end{proposition}

\begin{proof}
\step{Step 1. Recovery from the observation on a full time interval.}
For $L>0$, set

$$
u(s)=e^{-sN_\varepsilon}v,\qquad h(s)=Eu(s),
\qquad 0\le s\le L.
$$
We first prove
\begin{equation}\label{eq:endpoint-recovery}
|u(L)|
\le
C(1+L^{-q})\norm{h}_{L^\infty(0,L)}.
\end{equation}

For $|\eta|\le\varepsilon_0$, consider the observation Gramian

$$
G_\eta
=
\int_0^1
e^{-sN_\eta^\top}E^\top E e^{-sN_\eta}\,ds.
$$
If $v^\top G_\eta v=0$, then $Ee^{-sN_\eta}v=0$ on $[0,1]$.
Differentiating at $s=0$ gives $EN_\eta^rv=0$ for every $r\ge0$.
Taking $r=d_i$ and using \eqref{eq:coordinate-recovery}, we obtain
$v_i=0$ for every $i$. Hence $G_\eta$ is positive definite.
By continuity in $\eta$ and compactness of
$[-\varepsilon_0,\varepsilon_0]$, there is $\gamma>0$ such that
\begin{equation}\label{eq:unit-gramian}
\gamma|v|^2
\le
\int_0^1|Ee^{-sN_\eta}v|^2ds,
\qquad |\eta|\le\varepsilon_0.
\end{equation}

If $L\ge1$, apply \eqref{eq:unit-gramian} to $u(L-1)$. Since
$u(L-1+s)=e^{-sN_\varepsilon}u(L-1)$,

$$
\gamma|u(L-1)|^2
\le
\int_{L-1}^L|h(s)|^2\,ds
\le
\norm{h}_{L^\infty(0,L)}^2.
$$
The uniform bound for $e^{-N_\varepsilon}$ then gives
$|u(L)|\le C\norm{h}_{L^\infty(0,L)}$.

Now let $0<L<1$ and define
$\zeta(r)=D_Lu(Lr)$ for $0\le r\le1$. By
\eqref{eq:time-scaling},

$$
\zeta'(r)=-N_{L\varepsilon}\zeta(r),
\qquad
E\zeta(r)=h(Lr).
$$
Since $|L\varepsilon|\le\varepsilon_0$, we may apply
\eqref{eq:unit-gramian} with $\eta=L\varepsilon$. Together with the
uniform matrix bound on $[0,1]$, this gives

$$
|\zeta(1)|
\le
C|\zeta(0)|
\le
C\left(\int_0^1|h(Lr)|^2\,dr\right)^{1/2}
\le
C\norm{h}_{L^\infty(0,L)}.
$$
Since $u(L)=D_L^{-1}\zeta(1)$ and
$\norm{D_L^{-1}}=L^{-q}$, we obtain
$$
|u(L)|
\le
CL^{-q}\norm{h}_{L^\infty(0,L)}.
$$
This proves \eqref{eq:endpoint-recovery}.
\step{Step 2. Observation on a measurable time set.}
Let $I\subset[0,L]$ be an interval and let $G\subset I$ be measurable
with $|G|>0$. We claim that
\begin{equation}\label{eq:time-sup}
\norm{h}_{L^\infty(I)}
\le
C\left(\frac{|I|}{|G|}\right)^{n-1}
\norm{h}_{L^\infty(G)}.
\end{equation}

For $\varepsilon\ne0$, $N_\varepsilon$ is similar to
$\varepsilon K$. Since $K$ is skew-symmetric, each scalar component of
$h(s)=Ee^{-sN_\varepsilon}v$ is a linear combination of at most $n$
exponentials $c_ke^{i\tau_ks}$ with $\tau_k\in\R$.
Applying the Tur\'{a}n--Nazarov inequality
\cite{Nazarov1994}; see also
\cite[Theorem~1.1]{FriedlandYomdin2013},
componentwise and using the equivalence of norms in $\R^m$, we obtain
\eqref{eq:time-sup}, with a constant independent of the frequencies.
The result at $\varepsilon=0$ follows by continuity; in this case,
\eqref{eq:nilpotent-limit} also shows that $h$ is a polynomial of
degree at most $q$.

To pass from the $L^\infty$ observation to the $L^1$ observation,
consider the subset of $G$ on which

$$
|h(s)|
\le
\frac{2}{|G|}
\int_G|h(\tau)|\,d\tau.
$$
By Chebyshev's inequality, this subset has measure at least $|G|/2$.
Applying \eqref{eq:time-sup} to it yields
\begin{equation}\label{eq:time-remez}
\norm{h}_{L^\infty(I)}
\le
C\frac{|I|^{n-1}}{|G|^n}
\int_G|h(s)|\,ds.
\end{equation}

\step{Step 3. Restoring the eigenvalue and the heat factor.}
Take
$I=(0,\lambda\delta)$ and $G=\lambda F$.
Then $|G|=\lambda|F|$, and the change of variables $s=\lambda t$
gives
$$
\int_G|h(s)|\,ds
=
\lambda\int_F|h(\lambda t)|\,dt.
$$
Applying \eqref{eq:endpoint-recovery} with $L=\lambda\delta$ and then
\eqref{eq:time-remez}, we obtain
$$
\begin{aligned}
|u(\lambda\delta)|
\le
C\bigl(1+(\lambda\delta)^{-q}\bigr)
\frac{(\lambda\delta)^{n-1}}{(\lambda|F|)^n}
\lambda\int_F|h(\lambda t)|\,dt
=
C\frac{
\delta^{n-1}\bigl(1+(\lambda\delta)^{-q}\bigr)
}{
|F|^n
}
\int_F|h(\lambda t)|\,dt.
\end{aligned}
$$

Since
$
e^{-\lambda(aI_n+N_\varepsilon)\delta}
=
e^{-a\lambda\delta}e^{-\lambda\delta N_\varepsilon},
$
we have
$$
\left|
e^{-\lambda(aI_n+N_\varepsilon)\delta}v
\right|
=
e^{-a\lambda\delta}|u(\lambda\delta)|.
$$
Hence
$$
\left|
e^{-\lambda(aI_n+N_\varepsilon)\delta}v
\right|
\le
C\frac{
\delta^{n-1}\bigl(1+(\lambda\delta)^{-q}\bigr)
}{
|F|^n
}
\int_F
e^{-a\lambda\delta}|h(\lambda t)|\,dt.
$$
For $t\in F\subset(0,\delta)$,
$e^{-a\lambda\delta}\le e^{-a\lambda t}$, and
$$
e^{-a\lambda t}h(\lambda t)
=
Ee^{-\lambda(aI_n+N_\varepsilon)t}v.
$$
Therefore,
$$
\left|
e^{-\lambda(aI_n+N_\varepsilon)\delta}v
\right|
\le
C\frac{
\delta^{n-1}\bigl(1+(\lambda\delta)^{-q}\bigr)
}{
|F|^n
}
\int_F
\left|
Ee^{-\lambda(aI_n+N_\varepsilon)t}v
\right|\,dt.
$$
This proves \eqref{eq:modal-estimate}.
\end{proof}

\section{Integral interpolation and weighted observability}\label{sec:measurable}

In this section, we prove Theorem~\ref{thm:weighted}, which gives the weighted observability estimate on the measurable observation set $\D$. Throughout this section, we assume that \eqref{eq:kalman-rank} holds and fix $\varepsilon_0>0$.

Let \(L=-\Delta_D\) be the positive Dirichlet Laplacian, \(\{\phi_j\}_{j\ge1}\) be the real orthonormal basis of eigenfunctions such that
\begin{equation}\label{eq:eigenpairs}
L\phi_j=\lambda_j\phi_j,\qquad
0<\lambda_1<\lambda_2\le\cdots,\qquad
\lambda_j\longrightarrow\infty.
\end{equation}
Set \(H=L^2(\Omega;\R^n)\). If
$$
z=\sum_{j=1}^\infty z_j\phi_j\in H,
$$
then the solution of \eqref{eq:normalized-system} is given by
\begin{equation}\label{eq:normalized-semigroup}
S_\varepsilon(t)z
=
\sum_{j=1}^\infty
e^{-\lambda_j(aI_n+N_\varepsilon)t}z_j\phi_j.
\end{equation}
By \eqref{eq:matrix-decay}, \(S_\varepsilon(t)\) is a bounded strongly continuous semigroup on \(H\), uniformly for
\(|\varepsilon|\le\varepsilon_0\).

To distinguish the low-frequency modes from the rapidly decaying high-frequency modes, we introduce the corresponding spectral projection. For \(\Lambda\ge\lambda_1\), let \(\Pi_\Lambda\) denote the orthogonal projection onto the eigenspaces corresponding to
\(\lambda_j\le\Lambda\), acting componentwise. Since \(\Pi_\Lambda\)
acts only on the spatial variable, it commutes with constant matrices
and with \(S_\varepsilon(t)\). By \eqref{eq:matrix-decay},
\begin{equation}\label{eq:high-frequency}
\norm{S_\varepsilon(t)}\le M,\qquad
\norm{S_\varepsilon(t)(I-\Pi_\Lambda)z}_2
\le Me^{-c_0\Lambda t}\norm{z}_2,
\qquad c_0=\frac a2,
\end{equation}
for all \(t\ge0\), \(\Lambda\ge\lambda_1\), and
\(|\varepsilon|\le\varepsilon_0\).

For the low-frequency part, we use the spectral inequality from
\cite[Theorems~3 and~5]{ApraizEscauriazaWangZhang2014}.
Fix \(x_*\in\Omega\), \(0<R\le1\), and \(0<d_*\le |B_R|\), with
\(B_{4R}(x_*)\subset\Omega\). Then there exists
\(C=C(\Omega,R,d_*/|B_R|)>0\) such that
\begin{equation}\label{eq:spatial-spectral}
\norm{f}_2
\le
Ce^{C\sqrt\Lambda}
\int_\omega |f(x)|\,dx
\end{equation}
for every \(\R^m\)-valued \(f\in\operatorname{Ran}\Pi_\Lambda\) and
every measurable set \(\omega\subset B_R(x_*)\) satisfying
\(|\omega|\ge d_*\).

The constant is uniform over all such sets \(\omega\). Indeed, one may
apply the scalar estimate to a measurable subset of \(\omega\) of
measure \(d_*\). Applying the scalar estimate componentwise and using
$$
\left(
\sum_{k=1}^m
\left(\int_\omega |f_k|\,dx\right)^2
\right)^{1/2}
\le
\int_\omega |f(x)|\,dx
$$
gives the vector-valued estimate \eqref{eq:spatial-spectral}.

As a preparation for the proof of Theorem~\ref{thm:weighted}, we next establish an integral interpolation estimate. The proof follows the argument of \cite[Theorem~3.1]{FuWangYuZhu2026}, with constants uniform for \(|\varepsilon|\le\varepsilon_0\).

\begin{proposition}\label{prop:interpolation}
Fix \(B_R(x_*)\) and \(d_*\) as above. Let \(0<S_1<S_2\le T\), and let
\(F\subset(S_1,S_2)\) be measurable with \(\mu=|F|>0\).
Let \(\mathcal W\subset B_R(x_*)\times F\) be measurable, and set
$$
\omega_t=\{x\in B_R(x_*):(x,t)\in\mathcal W\},
\qquad t\in F.
$$
Assume that
$
|\omega_t|\ge d_*
$
for almost every \(t\in F\). Then there exists a constant $C>0$, depending only on
$\Omega,a,T,B,K,\varepsilon_0,R$ and $d_*/|B_R|$, such that
\begin{equation}\label{eq:interpolation}
\norm{S_\varepsilon(S_2)z}_2
\le
C\mu^{-n}e^{C/S_1}
\left(
\int_F\int_{\omega_t}
|ES_\varepsilon(t)z(x)|\,dx\,dt
\right)^{1/2}
\norm{z}_2^{1/2}
\end{equation}
for every \(z\in H\) and \(|\varepsilon|\le\varepsilon_0\).
\end{proposition}

\begin{proof}
Fix \(z\in H\) and \(|\varepsilon|\le\varepsilon_0\). By
\eqref{eq:high-frequency} and the Cauchy--Schwarz inequality,
$$
ES_\varepsilon(t)z\in
L^1\bigl(B_R(x_*)\times(0,T);\R^m\bigr).
$$
Hence Fubini's theorem applies to the integrals below. Throughout the
proof, \(C\) denotes a positive constant with the dependence stated in
the proposition, which may change from line to line.

First, we estimate the low-frequency part. Fix
\(\Lambda\ge\lambda_1\) and write
$$
\Pi_\Lambda S_\varepsilon(S_1)z
=
\sum_{\lambda_j\le\Lambda}v_j\phi_j,
\qquad
v_j=e^{-\lambda_j(aI_n+N_\varepsilon)S_1}z_j.
$$
Set
$
\delta=S_2-S_1.
$
Since \(F\subset(S_1,S_2)\), we have
$F-S_1\subset(0,\delta)$ and
$|F-S_1|=|F|=\mu$.
We may therefore apply Proposition~\ref{prop:modal} to each \(v_j\)
with the time set \(F-S_1\). Since \(q\le n-1\),
$$
\delta^{n-1}\bigl(1+(\lambda_j\delta)^{-q}\bigr)
=
\delta^{n-1}
+
\lambda_j^{-q}\delta^{n-1-q}
\le C
$$
for \(0<\delta\le T\) and \(\lambda_j\ge\lambda_1\). Hence
$$
\left|
e^{-\lambda_j(aI_n+N_\varepsilon)\delta}v_j
\right|
\le
C\mu^{-n}
\int_F
\left|
Ee^{-\lambda_j(aI_n+N_\varepsilon)(t-S_1)}v_j
\right|\,dt.
$$
Since
$$
e^{-\lambda_j(aI_n+N_\varepsilon)(t-S_1)}v_j
=
e^{-\lambda_j(aI_n+N_\varepsilon)t}z_j,
$$
the right-hand side is the observation of the original solution on the
\(j\)th mode. Taking the \(\ell^2\)-norm over
\(\lambda_j\le\Lambda\), and then using Minkowski's integral
inequality and Parseval's identity, we obtain
\begin{equation}\label{eq:low-temporal}
\norm{S_\varepsilon(S_2)\Pi_\Lambda z}_2
\le
C\mu^{-n}
\int_F
\norm{ES_\varepsilon(t)\Pi_\Lambda z}_2\,dt.
\end{equation}
The constant is independent of \(\Lambda\). For almost every \(t\in F\), the function
\(ES_\varepsilon(t)\Pi_\Lambda z\) belongs to
\(\operatorname{Ran}\Pi_\Lambda\). Applying
\eqref{eq:spatial-spectral} gives
\begin{equation}\label{eq:low-spatial}
\norm{S_\varepsilon(S_2)\Pi_\Lambda z}_2
\le
C\mu^{-n}e^{C\sqrt\Lambda}
\int_F\int_{\omega_t}
|ES_\varepsilon(t)\Pi_\Lambda z(x)|\,dx\,dt.
\end{equation}

Next, we estimate the high-frequency part. Since
$
\Pi_\Lambda z=z-(I-\Pi_\Lambda)z,
$
the observation term in \eqref{eq:low-spatial} is bounded by the full
observation and the contribution of \((I-\Pi_\Lambda)z\). By
\eqref{eq:high-frequency},
\begin{align*}
\int_F\int_{\omega_t}
|ES_\varepsilon(t)(I-\Pi_\Lambda)z(x)|\,dx\,dt
\le
|B_R|^{1/2}\norm{E}
\int_F
\norm{S_\varepsilon(t)(I-\Pi_\Lambda)z}_2\,dt
\le
Ce^{-c_0\Lambda S_1}\norm{z}_2.
\end{align*}
Moreover,
$$
\norm{S_\varepsilon(S_2)(I-\Pi_\Lambda)z}_2
\le
Me^{-c_0\Lambda S_2}\norm{z}_2.
$$
Combining these estimates with \eqref{eq:low-spatial} and the
terminal high-frequency estimate gives the following bound. Since
\(S_2>S_1\), \(\mu\le T\), and \(e^{b\sqrt\Lambda}\ge1\), the
terminal high-frequency term is absorbed into the second term on the
right-hand side after enlarging the constant. Thus
\begin{equation}\label{eq:low-high}
\norm{S_\varepsilon(S_2)z}_2
\le
C\mu^{-n}
\left[
e^{b\sqrt\Lambda}
\int_F\int_{\omega_t}
|ES_\varepsilon(t)z(x)|\,dx\,dt
+
e^{b\sqrt\Lambda-c_0\Lambda S_1}\norm{z}_2
\right],
\end{equation}
where \(b>0\) is independent of \(\Lambda\), \(z\), and
\(\varepsilon\).

Finally, let \(0<\theta\le1\) and choose
$
\Lambda
=
\lambda_1+\frac{4}{c_0S_1}\log\frac1\theta.
$
Using
$$
b\sqrt\Lambda
\le
\frac{c_0S_1}{4}\Lambda+\frac{b^2}{c_0S_1},
$$
and \(S_1\le T\), we obtain
$$
e^{b\sqrt\Lambda}
\le
Ce^{C/S_1}\theta^{-1}\quad
\text{and}
\quad
e^{b\sqrt\Lambda-c_0\Lambda S_1}
\le
Ce^{C/S_1}\theta^3
\le
Ce^{C/S_1}\theta.
$$
Hence \eqref{eq:low-high} gives
\begin{equation}\label{eq:interpolation-parameter}
\norm{S_\varepsilon(S_2)z}_2
\le
C\mu^{-n}e^{C/S_1}
\left[
\theta^{-1}
\int_F\int_{\omega_t}
|ES_\varepsilon(t)z(x)|\,dx\,dt
+
\theta\norm{z}_2
\right]
\end{equation}
for every \(0<\theta\le1\).

If
$
0<
\int_F\int_{\omega_t}
|ES_\varepsilon(t)z(x)|\,dx\,dt
<
\norm{z}_2,
$
we choose
$$
\theta
=
\left(
\frac{
\displaystyle
\int_F\int_{\omega_t}
|ES_\varepsilon(t)z(x)|\,dx\,dt
}{
\norm{z}_2
}
\right)^{1/2}.
$$
Then \eqref{eq:interpolation-parameter} gives
$$
\norm{S_\varepsilon(S_2)z}_2
\le
C\mu^{-n}e^{C/S_1}
\left(
\int_F\int_{\omega_t}
|ES_\varepsilon(t)z(x)|\,dx\,dt
\right)^{1/2}
\norm{z}_2^{1/2}.
$$

If
$
\int_F\int_{\omega_t}
|ES_\varepsilon(t)z(x)|\,dx\,dt
\ge
\norm{z}_2>0,
$
the uniform semigroup bound gives
$$
\norm{S_\varepsilon(S_2)z}_2
\le
M\norm{z}_2
\le
M
\left(
\int_F\int_{\omega_t}
|ES_\varepsilon(t)z(x)|\,dx\,dt
\right)^{1/2}
\norm{z}_2^{1/2}.
$$

If the observation integral is zero, we let
\(\theta\downarrow0\) in \eqref{eq:interpolation-parameter}. The case
\(z=0\) is immediate. Thus \eqref{eq:interpolation} follows.
\end{proof}

We now use Proposition~\ref{prop:interpolation} to prove the observability estimate on \(\D\). The argument follows the telescoping series method in
\cite{PhungWang2013,ApraizEscauriazaWangZhang2014,WangZhang2017};
see also \cite[Theorem~1.1]{FuWangYuZhu2026}.

\begin{proof}[Proof of Theorem~\ref{thm:weighted}]
We first prove the uniform observability estimate
\begin{equation}\label{eq:normalized-observability}
\norm{S_\varepsilon(T)z}_2
\le C\int_\D|ES_\varepsilon(t)z(x)|\,dx\,dt,
\qquad z\in H,\quad |\varepsilon|\le\varepsilon_0.
\end{equation}
The weighted estimate for \(y_\varepsilon\) will then follow from the
change of variables introduced above.

\step{Step 1. Select spatial sections of uniformly positive measure.}
By the Lebesgue density theorem, there exist a ball \(B_R(x_*)\) and
an interval \(I\subset(0,T)\) such that
$$
0<R\le1,\qquad
\overline{B_{4R}(x_*)}\subset\Omega,\qquad
\D_0:=\D\cap(B_R(x_*)\times I),\qquad
|\D_0|>0.
$$
For almost every \(t\in I\), set
$$
\D_{0,t}:=\{x\in B_R(x_*):(x,t)\in\D_0\},
\qquad
d_*:=\frac{|\D_0|}{2|I|},
\qquad
\mathcal E:=\{t\in I:|\D_{0,t}|\ge d_*\}.
$$
Since \(|\D_0|\le |B_R||I|\), we have \(0<d_*\le |B_R|\).
By Fubini's theorem,
$$
|\D_0|
=\int_I|\D_{0,t}|\,dt
\le d_*|I|+|B_R||\mathcal E|
=\frac{|\D_0|}{2}+|B_R||\mathcal E|.
$$
Therefore
\begin{equation}\label{eq:good-times}
|\mathcal E|
\ge \frac{|\D_0|}{2|B_R|}>0,
\qquad
|\D_{0,t}|\ge d_*
\quad\text{for a.e. }t\in\mathcal E.
\end{equation}
Thus the sections \(\D_{0,t}\), \(t\in\mathcal E\), satisfy the
spatial-section condition in Proposition~\ref{prop:interpolation}.

\step{Step 2. Apply the interpolation estimate on a sequence of time intervals.}
Choose a density point \(t_*\in\mathcal E\) in the interior of \(I\), and set
\(\kappa=\sqrt{3/2}\). Since \(t_*\) is a density point of
\(\mathcal E\), there exists \(r_0>0\) such that
\(t_*+r_0<\sup I\) and
\begin{equation}\label{eq:density-bound}
|(t_*,t_*+r)\setminus\mathcal E|
\le
\frac{1-\kappa^{-1}}{2}\,r,
\qquad 0<r\le r_0.
\end{equation}
For $j\ge1$, define
$t_j=t_*+r_0\kappa^{-(j-1)}$
and
$\delta_j=t_j-t_{j+1}.$
Then
$$
\delta_{j+1}=\frac{\delta_j}{\kappa},
\qquad
\delta_j=(1-\kappa^{-1})(t_j-t_*).
$$
Using \eqref{eq:density-bound} with \(r=t_j-t_*\), we obtain
\begin{equation}\label{eq:interval-density}
|\mathcal E\cap(t_{j+1},t_j)|
\ge
\frac{\delta_j}{2}.
\end{equation}
Fix \(z\in H\), and set
$$
H_j:=\norm{S_\varepsilon(t_j)z}_2,
\qquad
X_j:=
\int_{\mathcal E\cap(t_{j+1},t_j)}
\int_{\D_{0,t}}
|ES_\varepsilon(t)z(x)|\,dx\,dt.
$$
We apply Proposition~\ref{prop:interpolation} to
\(S_\varepsilon(t_{j+2})z\). After the change of variable
\(s=t-t_{j+2}\), take
$$
S_1=\delta_{j+1},
\qquad
S_2=\delta_j+\delta_{j+1},
\qquad
F=\bigl(\mathcal E\cap(t_{j+1},t_j)\bigr)-t_{j+2}.
$$
Since
$
S_2
=(1-\kappa^{-2})(t_j-t_*)
<r_0<T,
$
we have \(0<S_1<S_2\le T\). On \(B_R(x_*)\times F\), use the
shifted observation set
$
\bigl\{(x,s):(x,s+t_{j+2})\in\D_0\bigr\}.
$
Its spatial section at time \(s\in F\) is
\(\D_{0,s+t_{j+2}}\), whose measure is at least \(d_*\).
Moreover, by \eqref{eq:interval-density},
$$
|F|
=
|\mathcal E\cap(t_{j+1},t_j)|
\ge
\frac{\delta_j}{2}.
$$
Therefore,
$$
|F|^{-n}\le 2^n\delta_j^{-n},
\qquad
e^{C/S_1}
=e^{C/\delta_{j+1}}
=e^{C\kappa/\delta_j}.
$$
The observation term in Proposition~\ref{prop:interpolation} is
exactly \(X_j\), while the initial norm is \(H_{j+2}\). Hence
$$
H_j
\le
C\delta_j^{-n}e^{C/\delta_j}
\bigl(X_jH_{j+2}\bigr)^{1/2}.
$$

Since \(0<\delta_j\le T\) and
$
\sup_{0<s\le T}s^{-n}e^{-1/s}<\infty,
$
the factor \(\delta_j^{-n}\) can be absorbed into the exponential.
Hence there exists \(A>0\), independent of \(j\) and
\(\varepsilon\), such that
\begin{equation}\label{eq:pre-telescope}
H_j
\le
e^{A/\delta_j}
\bigl(X_jH_{j+2}\bigr)^{1/2}.
\end{equation}

\step{Step 3. Sum the estimates and return to the original variables.}
Multiplying \eqref{eq:pre-telescope} by \(e^{-4A/\delta_j}\) gives
$$
e^{-4A/\delta_j}H_j
\le
e^{-3A/\delta_j}
\bigl(X_jH_{j+2}\bigr)^{1/2}.
$$
Using
$
xy\le\frac14x^2+y^2
$
with \(x=X_j^{1/2}\) and
\(y=e^{-3A/\delta_j}H_{j+2}^{1/2}\), we obtain
$$
e^{-4A/\delta_j}H_j
\le
\frac14X_j+e^{-6A/\delta_j}H_{j+2}.
$$
Since \(\kappa^2=3/2\), we have
$
\delta_{j+2}=\frac23\delta_j$
and
$
e^{-6A/\delta_j}
=
e^{-4A/\delta_{j+2}}.
$
Hence
$$
e^{-4A/\delta_j}H_j
-
e^{-4A/\delta_{j+2}}H_{j+2}
\le
\frac14X_j.
$$
Summing over \(j=1,3,\ldots,2N-1\) gives
$$
e^{-4A/\delta_1}H_1
\le
\frac14\sum_{k=0}^{N-1}X_{2k+1}
+
e^{-4A/\delta_{2N+1}}H_{2N+1}.
$$

The time intervals appearing in the sum are disjoint, and therefore
$$
\sum_{k=0}^{N-1}X_{2k+1}
\le
\int_{\D_0}
|ES_\varepsilon(t)z(x)|\,dx\,dt.
$$
Moreover,
$$
H_{2N+1}\le M\norm{z}_2,
\qquad
\delta_{2N+1}\longrightarrow0.
$$
Therefore,
$$
0\le
e^{-4A/\delta_{2N+1}}H_{2N+1}
\le
M\norm{z}_2e^{-4A/\delta_{2N+1}}
\longrightarrow0.
$$
Letting \(N\to\infty\), we obtain
$$
\norm{S_\varepsilon(t_1)z}_2
\le
C\int_{\D_0}
|ES_\varepsilon(t)z(x)|\,dx\,dt.
$$
Since \(t_1<T\), the uniform semigroup bound and
\(\D_0\subset\D\) yield
$$
\norm{S_\varepsilon(T)z}_2
\le
C\int_\D
|ES_\varepsilon(t)z(x)|\,dx\,dt.
$$
This proves \eqref{eq:normalized-observability}, with \(C\) independent
of \(\varepsilon\).

It remains to return to the original variables. By
\eqref{eq:intertwining},
$$
e^{-\lambda_j(aI_n+N_\varepsilon)t}D_\varepsilon P
=
D_\varepsilon P
e^{-\lambda_j(aI_n+\varepsilon K)t}
$$
for every \(j\) and every real \(\varepsilon\). Using the spectral
decomposition, we obtain
$
S_\varepsilon(t)D_\varepsilon Py_0
=
D_\varepsilon Py_\varepsilon(t;y_0).
$
Taking \(z=D_\varepsilon Py_0\) in
\eqref{eq:normalized-observability} and using
\(ED_\varepsilon P=B\), we get
$$
\norm{D_\varepsilon P
y_\varepsilon(T;y_0)}_2
\le
C\int_\D
|By_\varepsilon(x,t;y_0)|\,dx\,dt.
$$

Finally, integrating the pointwise comparison
\eqref{eq:weight-comparison} over \(\Omega\) gives
\eqref{eq:weighted-observability}. Since
\eqref{eq:intertwining} and \eqref{eq:weight-comparison} also hold at
\(\varepsilon=0\), the proof is valid for
\(|\varepsilon|\le\varepsilon_0\).
\end{proof}

\section{Sharp costs and examples}\label{sec:sharpness}

In this section, we prove Theorem~\ref{thm:sharp}, which gives the exact blow-up rates of the directional and full-state observability costs as $\varepsilon\to0$. We then give several examples to illustrate how these rates depend on the Kalman depths.

\begin{proof}[Proof of Theorem~\ref{thm:sharp}]
We first prove the upper bounds. Assume \eqref{eq:kalman-rank}. Fix $\ell\ne0$ and let
$
r=\rho(\ell).
$
Since $\ell\in\Row_r$, there exist rows
$c_j\in\R^{1\times m}$, $0\le j\le r$, such that
$$
\ell=\sum_{j=0}^r c_jBK^j.
$$

For $0<|\varepsilon|\le1$, we have
\begin{align*}
|\varepsilon|^r\norm{\ell y_\varepsilon(T)}_2
&\le
\sum_{j=0}^r
|c_j|\,|\varepsilon|^{r-j}
\bigl(
|\varepsilon|^j\norm{BK^j y_\varepsilon(T)}_2
\bigr)\\
&\le
\left(\sum_{j=0}^r|c_j|^2\right)^{1/2}
\left(
\sum_{j=0}^q
|\varepsilon|^{2j}
\norm{BK^j y_\varepsilon(T)}_2^2
\right)^{1/2}.
\end{align*}
By Theorem~\ref{thm:weighted},
$$
|\varepsilon|^r\norm{\ell y_\varepsilon(T)}_2
\le
M_\ell
\int_\D
|By_\varepsilon(x,t;y_0)|\,dx\,dt.
$$
This proves the upper bound in \eqref{eq:functional-sharp}.

For the full state, the Kalman rank condition implies that there exists
$c>0$ such that
$$
\sum_{j=0}^q|BK^jv|^2\ge c^2|v|^2,
\qquad v\in\R^n.
$$
Since $0<|\varepsilon|\le1$,
\begin{align}
\left(
\sum_{j=0}^q
|\varepsilon|^{2j}
\norm{BK^j y_\varepsilon(T)}_2^2
\right)^{1/2}
\ge
|\varepsilon|^q
\left(
\sum_{j=0}^q
\norm{BK^j y_\varepsilon(T)}_2^2
\right)^{1/2}\notag
\ge
c|\varepsilon|^q
\norm{y_\varepsilon(T)}_2.
\label{eq:full-upper}
\end{align}
Combining this estimate with Theorem~\ref{thm:weighted} gives the upper bound in
\eqref{eq:full-sharp}.
For each fixed $\varepsilon\ne0$, the Kalman rank condition and the positivity of all the weights in \eqref{eq:weighted-observability} imply that the weighted terminal quantity is equivalent to $\norm{y_\varepsilon(T)}_2$. Hence Theorem~\ref{thm:weighted} yields full-state observability.

We next prove the lower bounds. Fix $\ell\ne0$ and let
$
r=\rho(\ell).
$
Since $\ell\notin\Row_{r-1}$, there exists
$v\in\R^n$ such that
\begin{equation}\label{eq:depth-vector}
BK^jv=0\quad(0\le j<r),
\qquad
\ell v\ne0.
\end{equation}
When $r=0$, the first set of conditions is empty. Moreover,
$BK^rv\ne0$. Indeed, otherwise every row in $\Row_r$ would annihilate
$v$, which contradicts $\ell\in\Row_r$ and $\ell v\ne0$.

Let $\phi_1$ be a first Dirichlet eigenfunction normalized by
$\norm{\phi_1}_2=1$, and take
$
y_0=\phi_1v.
$
Then
\begin{equation}\label{eq:first-mode}
y_\varepsilon(x,t)
=
e^{-a\lambda_1t}\phi_1(x)
e^{-\varepsilon\lambda_1tK}v.
\end{equation}
By Taylor's formula and \eqref{eq:depth-vector},
\begin{equation}\label{eq:depth-taylor}
Be^{-sK}v
=
\frac{(-s)^r}{r!}BK^rv
+
O(|s|^{r+1})
\qquad (s\to0).
\end{equation}
Hence, for $|s|$ sufficiently small,
$$
|Be^{-sK}v|
\le C_r|s|^r,
$$
where $|s|^0=1$ when $r=0$. For sufficiently small
$|\varepsilon|$, this estimate applies to
$s=\varepsilon\lambda_1t$ for every $0\le t\le T$. Therefore,
\begin{equation}\label{eq:test-observation}
\begin{aligned}
\int_\D
|By_\varepsilon(x,t;\phi_1v)|\,dx\,dt
\le
C_r|\varepsilon|^r\lambda_1^r
\int_\D
t^r e^{-a\lambda_1t}|\phi_1(x)|\,dx\,dt
\le
A_r|\varepsilon|^r.
\end{aligned}
\end{equation}
Here $A_r>0$ is independent of $\varepsilon$. At the terminal time,
$
\ell e^{-\varepsilon\lambda_1TK}v
=
\ell v+O(|\varepsilon|).
$
Since $\ell v\ne0$, there exists $\varepsilon_\ell>0$ such that
$$
\left|
\ell e^{-\varepsilon\lambda_1TK}v
\right|
\ge
\frac12|\ell v|,
\qquad
0<|\varepsilon|\le\varepsilon_\ell.
$$
Hence
$$
\norm{\ell y_\varepsilon(T)}_2
\ge
\frac12e^{-a\lambda_1T}|\ell v|.
$$
Combining this estimate with \eqref{eq:test-observation} and the definition of
$C_\ell(\varepsilon;T,\D)$, we obtain
$$
C_\ell(\varepsilon;T,\D)
\ge
\frac{
\frac12e^{-a\lambda_1T}|\ell v|
}{
A_r|\varepsilon|^r
}
\ge
c_\ell|\varepsilon|^{-r}.
$$
This proves the lower bound in \eqref{eq:functional-sharp}.

For the full-state cost, choose a nonzero vector
$v\in\R^n$ annihilated by $\Row_{q-1}$. Such a vector exists because
$\Row_{q-1}\ne\R^{1\times n}$ by the definition of $q$. Taking again
$y_0=\phi_1v$, the same argument gives
$$
\int_\D
|By_\varepsilon(x,t;\phi_1v)|\,dx\,dt
\le
A_q|\varepsilon|^q.
$$
Since $K$ is skew-symmetric,
$e^{-\varepsilon\lambda_1TK}$ is orthogonal, and hence
$
\norm{y_\varepsilon(T)}_2
=
e^{-a\lambda_1T}|v|.
$
Therefore,
$$
C_{\rm full}(\varepsilon;T,\D)
\ge
\frac{
e^{-a\lambda_1T}|v|
}{
A_q|\varepsilon|^q
}
\ge
c|\varepsilon|^{-q}.
$$
This proves the lower bound in \eqref{eq:full-sharp}. The same argument
also applies when $q=0$, since $\Row_{-1}={0}$.

It remains to prove the necessity of the Kalman rank condition. If
\eqref{eq:kalman-rank} fails, there exists $v\ne0$ such that
\[
BK^jv=0,\qquad j=0,\ldots,n-1.
\]
By the Cayley--Hamilton theorem, the same holds for every $j\ge0$.
Hence
\[
Be^{-sK}v=0,
\qquad s\in\R.
\]
Taking $y_0=\phi_1v$ in \eqref{eq:first-mode}, we obtain
\[
By_\varepsilon(x,t)=0
\qquad\text{on }\Omega\times(0,T),
\]
while, since $K$ is skew-symmetric,
\[
\norm{y_\varepsilon(T)}_2
=
e^{-a\lambda_1T}|v|>0.
\]
Therefore
$
C_{\rm full}(\varepsilon;T,\D)=+\infty
$
for every $\varepsilon\ne0$. Combined with the full-state observability obtained above under \eqref{eq:kalman-rank}, this proves the last statement of Theorem~\ref{thm:sharp}.
\end{proof}

We next give a spectral form of the Kalman rank condition that will be useful in the examples below. By the classical Popov--Belevitch--Hautus criterion
\cite{Hautus1969}, \eqref{eq:kalman-rank} is equivalent to
\begin{equation}\label{eq:PBH}
\ker B_\C\cap\ker(K-\mu I_n)=\{0\}
\quad\text{for every }\mu\in\sigma(K),
\end{equation}
where $K$ and $B$ are regarded as acting on complex vectors and
$B_\C$ denotes the complexification of $B$.

Indeed, if $v$ belongs to the intersection in \eqref{eq:PBH}, then
$B_\C K^jv=0$ for every $j\ge0$. Conversely, if
\eqref{eq:kalman-rank} fails, the nonzero complex subspace
\[
V=\bigcap_{j=0}^{n-1}\ker(B_\C K^j)
\]
is invariant under $K$ by the Cayley--Hamilton theorem. Hence the
restriction of $K$ to $V$ has an eigenvector $v\ne0$, which satisfies
$B_\C v=0$. Thus \eqref{eq:PBH} is equivalent to requiring $B_\C$ to
be injective on every complex eigenspace of $K$. In particular, each
eigenspace has dimension at most $m$.

For a scalar observation, the orthogonal normal form of $K$ makes
this condition explicit.

\begin{proposition}\label{prop:scalar-normal-form}
Suppose $m=1$. Let $Q$ be orthogonal and satisfy
\begin{equation}\label{eq:skew-normal-form}
Q^\top KQ=\diag(\omega_1J,\ldots,\omega_sJ,0_k),
\qquad
J=\begin{pmatrix}0&-1\\1&0\end{pmatrix},
\end{equation}
where $\omega_j>0$ and $n=2s+k$. Write
\[
BQ=(\beta_1,\ldots,\beta_s,\beta_0),
\qquad
\beta_j\in\R^{1\times2},
\quad
\beta_0\in\R^{1\times k}.
\]
Then the Kalman rank condition holds if and only if the frequencies
$\omega_1,\ldots,\omega_s$ are pairwise distinct, each $\beta_j$ is
nonzero, and $k\le1$, with $\beta_0\ne0$ when $k=1$.
\end{proposition}

\begin{proof}
The PBH condition \eqref{eq:PBH} is unchanged under the orthogonal
change of coordinates determined by $Q$. We may therefore work with
$Q^\top KQ$ and $BQ$.

For the block $\omega_jJ$, the complex eigenvalues are
$\pm i\omega_j$, with eigenvectors $(1,-i)^\top$ and $(1,i)^\top$,
respectively. If
$\beta_j=(\beta_{j,1},\beta_{j,2})$, their observed values are
$
\beta_{j,1}-i\beta_{j,2}$
and
$\beta_{j,1}+i\beta_{j,2}$.
Both are nonzero exactly when $\beta_j\ne0$.

If two frequencies are equal, the corresponding complex eigenspace
has dimension at least two. Since the observation is scalar,
$B_\C$ cannot be injective on this eigenspace. Hence the frequencies
must be pairwise distinct. Similarly, the eigenspace associated with
the eigenvalue $0$ has dimension $k$, so \eqref{eq:PBH} requires
$k\le1$; when $k=1$, its single direction must satisfy
$\beta_0\ne0$.

Conversely, under these conditions every complex eigenspace of $K$
is one-dimensional, and $B_\C$ is nonzero on each of them. Thus
\eqref{eq:PBH} holds, and hence so does the Kalman rank condition.
\end{proof}

\begin{corollary}\label{cor:scalar}
Suppose $m=1$ and \eqref{eq:kalman-rank} holds. Then
\[
q=n-1,\qquad
\rho(BK^r)=r,\qquad r=0,\ldots,n-1.
\]
For every $\varepsilon_0>0$, there exists $C>0$ such that
\begin{equation}\label{eq:scalar-weighted}
\left(
\sum_{r=0}^{n-1}
|\varepsilon|^{2r}
\norm{BK^r y_\varepsilon(T)}_2^2
\right)^{1/2}
\le
C\int_\D|By_\varepsilon|\,dx\,dt,
\qquad
|\varepsilon|\le\varepsilon_0.
\end{equation}
Moreover,
\begin{equation}\label{eq:scalar-full-cost}
C_{\rm full}(\varepsilon;T,\D)
\asymp
|\varepsilon|^{-(n-1)}
\qquad
(\varepsilon\to0,\ \varepsilon\ne0).
\end{equation}
\end{corollary}

\begin{proof}
Since $m=1$, the Kalman matrix in \eqref{eq:kalman-rank} has exactly
the $n$ rows
$
B,BK,\ldots,BK^{n-1}.
$
The rank condition therefore implies that these rows are linearly
independent. Hence
\[
\dim\Row_r=r+1,\qquad r=0,\ldots,n-1,
\]
which gives
\[
q=n-1,\qquad \rho(BK^r)=r.
\]
Estimate \eqref{eq:scalar-weighted} follows from
Theorem~\ref{thm:weighted}, and
\eqref{eq:scalar-full-cost} follows from
Theorem~\ref{thm:sharp}.
\end{proof}

We next return to the complex heat equation introduced in Section~\ref{sec:introduction}. For each fixed $\varepsilon\ne0$, this observability inequality was
proved in \cite[Theorem~1.1]{FuWangYuZhu2026}. Here we determine its
exact cost as $\varepsilon\to0$, together with the symmetry with respect
to the sign of $\varepsilon$ and an explicit lower bound.

\begin{corollary}\label{cor:complex}
Let $z_\varepsilon$ solve
\begin{equation}\label{eq:complex-heat}
\begin{cases}
\partial_t z_\varepsilon=(a+i\varepsilon)\Delta z_\varepsilon
   &\text{in }\Omega\times(0,T),\\
z_\varepsilon=0
   &\text{on }\partial\Omega\times(0,T),\\
z_\varepsilon(0)=z_0
   &\text{in }\Omega,
\end{cases}
\end{equation}
where $z_0\in L^2(\Omega;\C)$. Let
$C_{\mathrm{Re}}(\varepsilon;T,\D)$ denote the optimal constant in
\[
\norm{z_\varepsilon(T;z_0)}_2
\le
C_{\mathrm{Re}}(\varepsilon;T,\D)
\int_\D
|\mathrm{Re}\,z_\varepsilon(x,t;z_0)|\,dx\,dt,
\]
taken over all $z_0\in L^2(\Omega;\C)$. Then $C_{\mathrm{Re}}(\varepsilon;T,\D)<\infty$ for every
$\varepsilon\ne0$, and
\begin{equation}\label{eq:complex-symmetry}
C_{\mathrm{Re}}(-\varepsilon;T,\D)
=
C_{\mathrm{Re}}(\varepsilon;T,\D).
\end{equation}
Moreover,
\[
C_{\mathrm{Re}}(\varepsilon;T,\D)
\asymp
\frac1{|\varepsilon|}
\qquad
(\varepsilon\to0,\ \varepsilon\ne0).
\]

More precisely, let $\phi_1$ be a first Dirichlet eigenfunction such that
$\phi_1>0$ a.e. in $\Omega$ with
$\norm{\phi_1}_2=1$, and define
\begin{equation}\label{eq:complex-integral}
I_\D
=
\int_\D
t e^{-a\lambda_1t}\phi_1(x)\,dx\,dt.
\end{equation}
Then $0<I_\D<\infty$, and for every $0<\varepsilon_0\le1$ there exists
$C>0$ such that
\begin{equation}\label{eq:complex-cost}
\frac{e^{-a\lambda_1T}}
{\lambda_1I_\D}
\frac1{|\varepsilon|}
\le
C_{\mathrm{Re}}(\varepsilon;T,\D)
\le
\frac{C}{|\varepsilon|},
\qquad
0<|\varepsilon|\le\varepsilon_0.
\end{equation}
\end{corollary}

\begin{proof}
Write
$
z_\varepsilon=y_1+iy_2.
$
Then the corresponding real system has
\[
K=J,\qquad B=(1,0),\qquad BK=(0,-1).
\]
Hence Theorem~\ref{thm:weighted} gives
\begin{equation}\label{eq:complex-weighted}
\left(
\norm{y_1(T)}_2^2
+
|\varepsilon|^2\norm{y_2(T)}_2^2
\right)^{1/2}
\le
C\int_\D |y_1|\,dx\,dt.
\end{equation}
For $0<|\varepsilon|\le1$,
\[
\left(
\norm{y_1(T)}_2^2
+
|\varepsilon|^2\norm{y_2(T)}_2^2
\right)^{1/2}
\ge
|\varepsilon|\norm{z_\varepsilon(T)}_2.
\]
Therefore,
\[
C_{\mathrm{Re}}(\varepsilon;T,\D)
\le
\frac{C}{|\varepsilon|}.
\]

Next, if $z_\varepsilon$ solves \eqref{eq:complex-heat}, then
$\overline{z_\varepsilon}$ solves the same equation with $\varepsilon$
replaced by $-\varepsilon$. Moreover,
\[
\norm{\overline{z_\varepsilon}(T)}_2
=
\norm{z_\varepsilon(T)}_2,
\qquad
|\mathrm{Re}\,\overline{z_\varepsilon}|
=
|\mathrm{Re}\,z_\varepsilon|.
\]
It follows that
\[
C_{\mathrm{Re}}(-\varepsilon;T,\D)
\le
C_{\mathrm{Re}}(\varepsilon;T,\D).
\]
Replacing $\varepsilon$ by $-\varepsilon$ gives the reverse inequality,
and hence \eqref{eq:complex-symmetry}.

For the lower bound, take
$
z_0=i\phi_1.
$
Using \eqref{eq:intro-complex-mode}, we have
\[
\norm{z_\varepsilon(T)}_2
=
e^{-a\lambda_1T},
\]
while
\[
|\mathrm{Re}\,z_\varepsilon(x,t)|
=
e^{-a\lambda_1t}
|\sin(\varepsilon\lambda_1t)|
\phi_1(x).
\]
Since $|\sin s|\le |s|$,
\[
\int_\D
|\mathrm{Re}\,z_\varepsilon|\,dx\,dt
\le
|\varepsilon|\lambda_1
\int_\D
t e^{-a\lambda_1t}\phi_1(x)\,dx\,dt
=
|\varepsilon|\lambda_1 I_\D.
\]
Therefore, by the definition of
$C_{\mathrm{Re}}(\varepsilon;T,\D)$,
\[
C_{\mathrm{Re}}(\varepsilon;T,\D)
\ge
\frac{e^{-a\lambda_1T}}
{|\varepsilon|\lambda_1I_\D}.
\]
This proves \eqref{eq:complex-cost}.

Finally, $I_\D<\infty$ follows from the integrability of $\phi_1$ on
the bounded domain, while $I_\D>0$ follows from
$\phi_1>0$ in $\Omega$, $t>0$ on $(0,T)$, and $|\D|>0$.
\end{proof}

In this example, the rescaling introduced in Section~\ref{sec:modal} takes the explicit form
\begin{equation}\label{eq:complex-scaling}
w_\varepsilon=(y_1,-\varepsilon y_2)^\top,
\qquad
N_\varepsilon=
\begin{pmatrix}
0&1\\
-\varepsilon^2&0
\end{pmatrix},
\qquad
N_0=
\begin{pmatrix}
0&1\\
0&0
\end{pmatrix}.
\end{equation}
Thus $E=(1,0)$ and $EN_0=(0,1)$. For $\varepsilon\ne0$,
\begin{equation}\label{eq:complex-normalized-exponential}
e^{-\lambda(aI_2+N_\varepsilon)t}
=
e^{-a\lambda t}
\begin{pmatrix}
\cos(\varepsilon\lambda t)
&
-\dfrac{\sin(\varepsilon\lambda t)}{\varepsilon}
\\[2mm]
\varepsilon\sin(\varepsilon\lambda t)
&
\cos(\varepsilon\lambda t)
\end{pmatrix},
\end{equation}
and, as $\varepsilon\to0$,
\[
e^{-\lambda(aI_2+N_\varepsilon)t}
\longrightarrow
e^{-a\lambda t}
\begin{pmatrix}
1&-\lambda t\\
0&1
\end{pmatrix}.
\]
The uniform weighted estimate therefore controls $-\varepsilon y_2$ together with $y_1$. Recovering the unweighted component $y_2$ introduces the factor $|\varepsilon|^{-1}$, which agrees with the sharp cost in Corollary~\ref{cor:complex}.

The final two examples illustrate the values of $q$ for one and two observation channels.

\begin{example}\label{ex:three}
Let
\[
K=
\begin{pmatrix}
0&-1&0\\
1&0&-1\\
0&1&0
\end{pmatrix},
\qquad
B=(1,0,0).
\]
Then
\[
B=(1,0,0),\qquad
BK=(0,-1,0),\qquad
BK^2=(-1,0,1).
\]
These three rows are linearly independent. Hence
\[
q=2.
\]
Theorem~\ref{thm:weighted} gives
\begin{equation}\label{eq:three-weighted}
\left(
\norm{y_1(T)}_2^2
+
|\varepsilon|^2\norm{y_2(T)}_2^2
+
|\varepsilon|^4\norm{y_3(T)-y_1(T)}_2^2
\right)^{1/2}
\le
C\int_\D|y_1|\,dx\,dt,
\end{equation}
uniformly for $|\varepsilon|\le\varepsilon_0$, where
$\varepsilon_0>0$ is fixed.

Moreover,
\[
\rho(e_1^\top)=0,\qquad
\rho(e_2^\top)=1,\qquad
\rho(e_3^\top)=2.
\]
Thus Theorem~\ref{thm:sharp} shows that the observability costs of
$y_1(T)$, $y_2(T)$, and $y_3(T)$ have orders $1$, $|\varepsilon|^{-1}$ and $|\varepsilon|^{-2}$ respectively. In particular,
\[
C_{\rm full}(\varepsilon;T,\D)
\asymp|\varepsilon|^{-2}.
\]
\end{example}

\begin{example}\label{ex:four}
Let
\[
J=
\begin{pmatrix}
0&-1\\
1&0
\end{pmatrix},
\qquad
K=\diag(J,2J),
\qquad
B=
\begin{pmatrix}
1&0&0&0\\
0&0&1&0
\end{pmatrix}.
\]
Then
\[
BK=
\begin{pmatrix}
0&-1&0&0\\
0&0&0&-2
\end{pmatrix}.
\]
The rows of $B$ and $BK$ span $\R^{1\times4}$, and hence
\[
q=1.
\]
Theorem~\ref{thm:weighted} gives
\[
\left(
\norm{y_1(T)}_2^2
+
\norm{y_3(T)}_2^2
+
|\varepsilon|^2\norm{y_2(T)}_2^2
+
4|\varepsilon|^2\norm{y_4(T)}_2^2
\right)^{1/2}
\le
C\int_\D
\left|(y_1,y_3)\right|\,dx\,dt,
\]
uniformly for $|\varepsilon|\le\varepsilon_0$, where
$\varepsilon_0>0$ is fixed.

Here
\[
\rho(e_1^\top)=\rho(e_3^\top)=0,
\qquad
\rho(e_2^\top)=\rho(e_4^\top)=1.
\]
Therefore, Theorem~\ref{thm:sharp} gives
\[
C_{\rm full}(\varepsilon;T,\D)
\asymp
|\varepsilon|^{-1}.
\]
Here $n=4$ and $m=2$: the two observation channels directly observe
$y_1$ and $y_3$, while $y_2$ and $y_4$ are reached after one coupling
step. Thus this four-component system has maximal Kalman depth $q=1$, although the state dimension is 4.
\end{example}

\section{Minimum-norm null controls}\label{sec:control}
We next use duality to relate the full-state observability cost to the
minimum $L^\infty$-norm of null controls. Let
$\Q\subset\Omega\times(0,T)$ be measurable, with $|\Q|>0$, and consider
\begin{equation}\label{eq:controlled-system}
\begin{cases}
\partial_t x_\varepsilon=(aI_n-\varepsilon K)\Delta x_\varepsilon
 +\ind_\Q B^\top u&\text{in }\Omega\times(0,T),\\
x_\varepsilon=0&\text{on }\partial\Omega\times(0,T),\\
x_\varepsilon(0)=x_0&\text{in }\Omega.
\end{cases}
\end{equation}
Here $x_0\in H$ and $u\in L^\infty(\Q;\R^m)$, extended by zero
outside $\Q$. For each $x_0\in H$, define
\begin{equation}\label{eq:individual-control-cost}
\mathcal N_\varepsilon(x_0)
=\inf\bigl\{\norm{u}_{L^\infty(\Q)}:
u\in L^\infty(\Q;\R^m),\quad
x_\varepsilon(T;x_0,u)=0\bigr\},
\end{equation}
where the infimum of the empty set is understood as $+\infty$. The
worst-case control cost is
\begin{equation}\label{eq:control-cost-definition}
\controlcost(\varepsilon;T,\Q)
=\sup_{\norm{x_0}_2=1}\mathcal N_\varepsilon(x_0).
\end{equation}
We also introduce the time-reflected set
$
\Q^\sharp=\{(x,s):(x,T-s)\in\Q\}.
$

The sign in \eqref{eq:controlled-system} is chosen so that its adjoint
diffusion matrix is
\[
(aI_n-\varepsilon K)^\top=aI_n+\varepsilon K,
\]
which is exactly the diffusion matrix in \eqref{eq:original-system}.
After reversing time in the adjoint equation, the observation on $\Q$
becomes an observation on $\Q^\sharp$. Thus the observability results
of the previous sections apply directly to the null control problem.

\begin{theorem}\label{thm:control}
Assume \eqref{eq:kalman-rank}. For every $\varepsilon\ne0$ and
$x_0\in H$, there exists a null control
$u_\varepsilon^*\in L^\infty(\Q;\R^m)$ such that
$
\norm{u_\varepsilon^*}_{L^\infty(\Q)}
=\mathcal N_\varepsilon(x_0).
$
Moreover, the worst-case minimum control cost satisfies
\begin{equation}\label{eq:control-identity}
\controlcost(\varepsilon;T,\Q)
=C_{\rm full}(\varepsilon;T,\Q^\sharp).
\end{equation}
Consequently, there exist $c,C,\varepsilon_*>0$ such that
\begin{equation}\label{eq:control-sharp}
c|\varepsilon|^{-q}
\le \controlcost(\varepsilon;T,\Q)
\le C|\varepsilon|^{-q},
\qquad
0<|\varepsilon|\le\varepsilon_*.
\end{equation}
\end{theorem}

\begin{proof}
All spaces and pairings are real. We use $L^1$--$L^\infty$ duality
with preservation of norms; see \cite{TucsnakWeiss2009} for the
general observability and control framework.

Let $T_\varepsilon(t)\phi=y_\varepsilon(t;\phi)$ denote the
homogeneous semigroup associated with \eqref{eq:original-system}.
Since $K^\top=-K$, its spectral representation gives
$
T_{-\varepsilon}(t)^*=T_\varepsilon(t).
$
Moreover, both semigroups are contractions on $H$, since
$e^{-\varepsilon\lambda_jKt}$ is orthogonal. Define
\begin{equation}\label{eq:control-map}
W_\varepsilon u
=\int_0^T T_{-\varepsilon}(T-t)
 \bigl(\ind_\Q(\cdot,t)B^\top u(\cdot,t)\bigr)\,dt.
\end{equation}
Writing $\Q_t=\{x:(x,t)\in\Q\}$, we obtain
\[
\norm{W_\varepsilon u}_2
\le \norm{B}\norm{u}_{L^\infty(\Q)}
       \int_0^T|\Q_t|^{1/2}\,dt
\le \norm{B}\sqrt{T|\Q|}\norm{u}_{L^\infty(\Q)}.
\]
Hence
\begin{equation}\label{eq:null-control-condition}
x_\varepsilon(T;x_0,u)
=T_{-\varepsilon}(T)x_0+W_\varepsilon u.
\end{equation}

For $\phi\in H$, set
\[
(\mathcal O_\varepsilon\phi)(x,t)
=BT_\varepsilon(T-t)\phi(x),
\qquad (x,t)\in\Q.
\]
Then
\[
\norm{\mathcal O_\varepsilon\phi}_{L^1(\Q)}
\le \norm{B}\sqrt{T|\Q|}\norm{\phi}_2.
\]
Using $T_{-\varepsilon}(t)^*=T_\varepsilon(t)$ and Fubini's theorem,
we have
\begin{equation}\label{eq:control-pairing}
\langle W_\varepsilon u,\phi\rangle_H
=\int_\Q u\cdot\mathcal O_\varepsilon\phi\,dx\,dt.
\end{equation}
After the change of variables $s=T-t$,
\begin{equation}\label{eq:reflected-observation}
\norm{\mathcal O_\varepsilon\phi}_{L^1(\Q)}
=\int_{\Q^\sharp}|BT_\varepsilon(s)\phi(x)|\,dx\,ds.
\end{equation}
Therefore, Theorem~\ref{thm:sharp} gives
\begin{equation}\label{eq:dual-observability}
\norm{T_\varepsilon(T)\phi}_2
\le C_{\rm full}(\varepsilon;T,\Q^\sharp)
       \norm{\mathcal O_\varepsilon\phi}_{L^1(\Q)}.
\end{equation}
Finally, $T_\varepsilon(T)$ is injective, since every matrix
exponential in its spectral representation is invertible. If
$\mathcal O_\varepsilon\phi=0$, then
\eqref{eq:dual-observability} gives $T_\varepsilon(T)\phi=0$.
The injectivity of $T_\varepsilon(T)$ therefore implies $\phi=0$.
Hence $\mathcal O_\varepsilon\phi\ne0$ for every $\phi\ne0$.

Fix $x_0\in H$. On the subspace
$\mathcal O_\varepsilon(H)\subset L^1(\Q;\R^m)$, define
\[
\mathcal L_{x_0}(\mathcal O_\varepsilon\phi)
=-\langle x_0,T_\varepsilon(T)\phi\rangle_H.
\]
Since $\mathcal O_\varepsilon$ is injective, this functional is
well defined. Moreover, \eqref{eq:dual-observability} gives
\[
|\mathcal L_{x_0}(\mathcal O_\varepsilon\phi)|
\le
\norm{x_0}_2
C_{\rm full}(\varepsilon;T,\Q^\sharp)
\norm{\mathcal O_\varepsilon\phi}_{L^1(\Q)},
\]
so $\mathcal L_{x_0}$ is bounded. Its norm is
\begin{equation}\label{eq:functional-norm}
\norm{\mathcal L_{x_0}}
=\sup_{\phi\ne0}
\frac{|\langle x_0,T_\varepsilon(T)\phi\rangle_H|}
     {\norm{\mathcal O_\varepsilon\phi}_{L^1(\Q)}}.
\end{equation}

By the Hahn--Banach theorem, $\mathcal L_{x_0}$ extends to a bounded
functional on $L^1(\Q;\R^m)$ with the same norm. The
$L^1$--$L^\infty$ duality then gives
$u_{x_0}\in L^\infty(\Q;\R^m)$ such that
\[
\int_\Q u_{x_0}\cdot\mathcal O_\varepsilon\phi\,dx\,dt
=-\langle x_0,T_\varepsilon(T)\phi\rangle_H,
\qquad
\norm{u_{x_0}}_{L^\infty(\Q)}
=\norm{\mathcal L_{x_0}}.
\]
Together with \eqref{eq:control-pairing}, this identity yields
\[
W_\varepsilon u_{x_0}
=-T_{-\varepsilon}(T)x_0.
\]
Hence \eqref{eq:null-control-condition} gives
$x_\varepsilon(T;x_0,u_{x_0})=0$.

Now let $u$ be any null control for $x_0$. By
\eqref{eq:null-control-condition} and \eqref{eq:control-pairing},
\[
|\langle x_0,T_\varepsilon(T)\phi\rangle_H|
\le
\norm{u}_{L^\infty(\Q)}
\norm{\mathcal O_\varepsilon\phi}_{L^1(\Q)}
\]
for every $\phi\in H$. Taking the supremum over $\phi\ne0$ and using
\eqref{eq:functional-norm}, we obtain
\[
\norm{\mathcal L_{x_0}}
\le\norm{u}_{L^\infty(\Q)}.
\]
Since the control $u_{x_0}$ attains this bound, it is a minimum-norm
null control. Therefore,
\begin{equation}\label{eq:minimum-control-formula}
\mathcal N_\varepsilon(x_0)
=\sup_{\phi\ne0}
\frac{|\langle x_0,T_\varepsilon(T)\phi\rangle_H|}
     {\norm{\mathcal O_\varepsilon\phi}_{L^1(\Q)}}.
\end{equation}

Finally, taking the supremum in
\eqref{eq:minimum-control-formula} over $\norm{x_0}_2=1$ and using
\eqref{eq:reflected-observation}, we obtain
\begin{align*}
\controlcost(\varepsilon;T,\Q)
&=\sup_{\phi\ne0}
\frac{\displaystyle\sup_{\norm{x_0}_2=1}
|\langle x_0,T_\varepsilon(T)\phi\rangle_H|}
{\norm{\mathcal O_\varepsilon\phi}_{L^1(\Q)}}\\
&=\sup_{\phi\ne0}
\frac{\norm{T_\varepsilon(T)\phi}_2}
{\displaystyle\int_{\Q^\sharp}
|BT_\varepsilon(s)\phi(x)|\,dx\,ds}\\
&=C_{\rm full}(\varepsilon;T,\Q^\sharp).
\end{align*}
The last equality follows from the definition of the optimal
full-state observability constant. Hence
\eqref{eq:control-identity} holds. Since
$|\Q^\sharp|=|\Q|>0$, \eqref{eq:control-sharp} follows from
Theorem~\ref{thm:sharp}.
\end{proof}

%

\end{document}